%% file: KL_Divergence.tex
\documentclass{article}
\usepackage{graphicx,amsthm,amssymb,amsmath}
\usepackage[pdftex]{hyperref}
\hypersetup{
 colorlinks=true,
 urlcolor=blue
}

\newtheorem{theorem}{Theorem} %Defines \begin{theorem} to write "Theorem"
\theoremstyle{definition}

\newtheorem{definition}{Definition}

\theoremstyle{remark}

\author{Dashiell E.A.\,Fryer}
\title{The Kullback-Leibler Divergence as a Lyapunov Function for Incentive Based Game Dynamics}
\date{\today}

\begin{document}
\maketitle

\begin{abstract}
It has been shown that the Kullback-Leibler divergence is a Lyapunov function for the replicator equations at evolutionary stable states, or ESS. In this paper we extend the result to a more general class of game dynamics. As a result, sufficient conditions can be given for the asymptotic stability of rest points for the entire class of incentive dynamics. The previous known results will be can be shown as corollaries to the main theorem. 
\end{abstract}

\input{IncentiveStableStates}

\bibliographystyle{amsalpha}
\bibliography{refs}
\end{document}

%% file: IncentiveStableStates.tex
\section{Information Theory and The Replicator Dynamics}

Information theory was originally developed by Claude Shannon and Warren Weaver~\cite{shannon2001mathematical,shannon1949mathematical} as a mathematical framework to describe problems in communication including, but not limited to, data compression and storage. They introduced measures of information called entropy\footnote{In fact, the Shannon entropy is simply the Boltzmann entropy~\cite{jaynes1965gibbs} without the constants}. Shannon's entropy, denoted $H(P)$, is a measure of the average uncertainty in a random variable, $P$. It can be interpreted as the average number of bits needed to encode a message drawn i.i.d. from P. Maximizing the entropy can be used to give a lower bound on this average number of bits needed for encryption.

For our purposes, the concepts of cross entropy and relative entropy will be of great use. The Kullback-Leibler divergence (KL divergence or $D_{KL}$)~\cite{kullback1951information}, or relative entropy is a measure of information gain (loss) from one state to another. More precisely, it is an average measure of the additional bits needed to store $y$ given a code optimized to store $x$. It is defined as \begin{align*} D_{KL}(x||y) & = \sum_\alpha x_\alpha \ln \frac{x_\alpha}{y_\alpha} \\ & = \sum_\alpha x_\alpha\ln x_\alpha - \sum_\alpha x_\alpha\ln y_\alpha \\ & = H(x) - H(x,y). \end{align*} where $H(x,y)$ is the cross entropy of $x$ and $y$. It should be clear that minimizing $D_{KL}$ with respect to $y$ is equivalent to minimizing the cross entropy term as well. Intuitively, this is trying to find the best distribution to approximate the `true' distribution $x$ and is well known as the Principle of Minimum Discrimination or Minimum Discrimination Information.

Recall the definition of an evolutionary stable state or ESS~\cite{smith1974theory}. \begin{definition} A strategy profile $\hat{x}\in\Delta$ is an ESS if and only if $u(\hat{x},x) > u(x,x)$ for every $x\neq \hat{x}$ in a neighborhood of $\hat{x}$.
\end{definition}
In this context there is a single population playing a symmetric game against itself. It has been shown by Weibull \cite{weibull1997evolutionary}, and Harper \cite{harper2011escort} that the KL divergence is a Lyapunov function for the replicator equation at an ESS\footnote{This result continues to be true for $n$-population games.}. Further connection between evolutionary games and information theory can be realized by expanding the KL divergence in a Taylor series along $x=y$ and noting that the Hessian term is positive definite and is thus a metric. The derived metric, a localization of the global divergence, is called the Shahshahani metric~\cite{shahshahani1979new} and it has been shown that the replicator dynamics are gradient flows of this metric~\cite{hofbauer1998evolutionary}. 

\section{Incentive Stable States}

In \cite{fryer2012existence}, the notion of an incentive for a game was developed. Furthermore, a family of game dynamics was derived from these incentives and shown to be fully general in the following sense: any valid game dynamic can be achieved from an appropriate choice of incentive. To recap: an incentive, $\varphi(x)$, is valid if and only if, for every $i$ and $\alpha$, $x_{i\alpha} = 0 \Rightarrow \varphi_{i\alpha}(x) \geq 0$ and $\sum_\alpha \varphi_{i\alpha}(x) \neq -1$. The corresponding incentive dynamic is then given by \[\dot{x}_{i\alpha} = \varphi_{i\alpha}(x)-x_{i\alpha}\sum_\beta \varphi_{i\beta}(x).\]

The deep connections between information theory and the replicator dynamics lead us to believe that some of these properties are more general. Unfortunately, most of our incentive dynamics are not gradient flows of some Riemannian metric, but the Principle of Minimum Discrimination is compelling enough for us to believe we may be able to describe asymptotically stable states for the incentive dynamics. We begin by defining a notion of incentive stability that is closely related to the notion of ESS. 
\begin{definition} A strategy profile $\hat{x}$ is an incentive stable state or ISS if and only if \[x_i \cdot \frac{\varphi_i(x)}{x_i} < \hat{x}_i \cdot \frac{\varphi_i(x)}{x_i},\quad\forall i\] for $x\neq\hat{x}$ in a neighborhood of $\hat{x}$.
\end{definition}
The interpretation is exactly the same as in the ESS case: $\hat{x}$ is preferred to all distributions sufficiently close.

We can now show that all ISS are asymptotically stable for the corresponding incentive dynamics. Note: if there is only one agent we have a necessary and sufficient condition for the Kullback-Liebler divergence to be a strict Lyapunov function.

\begin{theorem}
If the state $\hat{x}$ is an interior incentive stable state for the corresponding incentive dynamics, then $\sum_i D_{KL}(\hat{x_i}||x_i)$ is a local Lyapunov function.
\end{theorem}
\begin{proof}
Define $V_i(x)=D_{KL}(\hat{x_i}||x_i)$ and $V(x) = \sum_i V_i(x)$. Then we have the following:
\begin{align*}
\dot{V_i}(x) & = -\sum_\alpha \hat{x}_{i\alpha}\frac{\dot{x}_{i\alpha}}{x_{i\alpha}} \\
			& = - \sum_\alpha \frac{\hat{x}_{i\alpha}}{x_{i\alpha}}\left[\varphi_{i\alpha}(x)-x_{i\alpha}\sum_\beta \varphi_{i\beta}(x)\right] \\
			& = \sum_\beta \varphi_{i\beta}(x) \sum_\alpha \hat{x}_{i\alpha} - \sum_\alpha \frac{\hat{x}_{i\alpha}}{x_{i\alpha}}\varphi_{i\alpha}(x) \\
			& = \sum_\alpha \frac{x_{i\alpha}-\hat{x}_{i\alpha}}{x_{i\alpha}}\varphi_{i\alpha}(x) < 0\\
			& \Leftrightarrow x_i \cdot \frac{\varphi_i(x)}{x_i} < \hat{x}_i \cdot \frac{\varphi_i(x)}{x_i}
\end{align*} 
\end{proof}

\section{Examples}

The replicator dynamics can be achieved using the incentive \[\varphi_{i\alpha}(x) = x_{i\alpha}(f_{i\alpha}(x) + g_i(x)),\] where $f$ is the fitness landscape, and $g$ is an arbitrary function. Thus the ISS condition trivially reduces to the familiar ESS condition.

There are however, many other dynamics that are asymptotically stable at ESS. For example, Nagurney and Zhang \cite{nagurney1996projected}, Lahkar and Sandholm\cite{lahkar2008projection}, and Harper \cite{harper2011escort} all show independently that the euclidean distance can be used as a Lyapunov function to establish this fact for the projection dynamics. 

\subsection{Best Reply}

The Best Reply dynamic as defined by Young \cite{young2001individual} has a simple incentive function, $\varphi_{i\alpha}(x) = BR_{i\alpha}(x)$, where $BR_{i\alpha}(x)=1$ if $e_{i\alpha}$ is a best reply to the current state $x$. A tiebreaker is assumed for instances where there is more than one best reply. Thus the incentive dynamic is \[ \dot{x}_{i\alpha} = BR_{i\alpha}(x) - x_{i\alpha}.\]

The ISS condition is rather simple to interpret:
\begin{align*}
x_{i} \cdot \frac{ \varphi_i(x) }{x_i} & < \hat{x}_i \cdot \frac{ \varphi_i(x) }{ x_i } \\
\sum_\alpha BR_{i\alpha}(x) & < \sum_{\alpha} \frac{ \hat{x}_{i\alpha} BR_{i\alpha}(x) }{ x_{i\alpha} } \\
1 & < \frac{\hat{x}_{i\beta}}{x_{i\beta}}, \text{where $e_{i\beta}$ is the best reply to $x$} \\
\Rightarrow x_{i\beta} & < \hat{x}_{i\beta}
\end{align*}

This occurs trivially at an ESS.

%% file: KL_Divergence.bbl
\newcommand{\etalchar}[1]{$^{#1}$}
\providecommand{\bysame}{\leavevmode\hbox to3em{\hrulefill}\thinspace}
\providecommand{\MR}{\relax\ifhmode\unskip\space\fi MR }
\providecommand{\MRhref}[2]{%
  \href{http://www.ams.org/mathscinet-getitem?mr=#1}{#2}
}
\providecommand{\href}[2]{#2}
\begin{thebibliography}{You01}

\bibitem[Fry12]{fryer2012existence}
D.E.A. Fryer, \emph{On the existence of general equilibrium in finite games and
  general game dynamics}, Arxiv preprint arXiv:1201.2384 (2012).

\bibitem[Har11]{harper2011escort}
M.~Harper, \emph{Escort evolutionary game theory}, Physica D: Nonlinear
  Phenomena (2011).

\bibitem[HS98]{hofbauer1998evolutionary}
J.~Hofbauer and K.~Sigmund, \emph{Evolutionary games and population dynamics},
  Cambridge Univ Pr, 1998.

\bibitem[Jay65]{jaynes1965gibbs}
E.T. Jaynes, \emph{Gibbs vs boltzmann entropies}, American Journal of Physics
  \textbf{33} (1965), 391.

\bibitem[KL51]{kullback1951information}
S.~Kullback and R.A. Leibler, \emph{On information and sufficiency}, The Annals
  of Mathematical Statistics \textbf{22} (1951), no.~1, 79--86.

\bibitem[LS08]{lahkar2008projection}
R.~Lahkar and W.H. Sandholm, \emph{The projection dynamic and the geometry of
  population games}, Games and Economic Behavior \textbf{64} (2008), no.~2,
  565--590.

\bibitem[NZ96]{nagurney1996projected}
A.~Nagurney and D.~Zhang, \emph{Projected dynamical systems and variational
  inequalities with applications}, Kluwer Academic Publishers, 1996.

\bibitem[S{\etalchar{+}}74]{smith1974theory}
M.~Smith et~al., \emph{The theory of games and the evolution of animal
  conflicts}, Journal of theoretical biology \textbf{47} (1974), no.~1,
  209--221.

\bibitem[Sha79]{shahshahani1979new}
S.~Shahshahani, \emph{A new mathematical framework for the study of linkage and
  selection}, vol. 211, Amer Mathematical Society, 1979.

\bibitem[Sha01]{shannon2001mathematical}
C.E. Shannon, \emph{A mathematical theory of communication}, ACM SIGMOBILE
  Mobile Computing and Communications Review \textbf{5} (2001), no.~1, 3--55.

\bibitem[SW49]{shannon1949mathematical}
C.E. Shannon and W.~Weaver, \emph{The mathematical theory of communication
  (urbana, il)}, University of Illinois Press \textbf{184} (1949), 1032--54.

\bibitem[Wei97]{weibull1997evolutionary}
J.W. Weibull, \emph{Evolutionary game theory}, The MIT press, 1997.

\bibitem[You01]{young2001individual}
H.P. Young, \emph{Individual strategy and social structure: An evolutionary
  theory of institutions}, Princeton Univ Pr, 2001.

\end{thebibliography}
